\PassOptionsToPackage{dvipsnames}{xcolor}
\documentclass[a4paper,twoside,reqno]{amsart}
\usepackage[T1]{fontenc}
\usepackage[english]{babel}
\usepackage{microtype}
\usepackage[left=2.8cm,right=2.8cm]{geometry}

\usepackage{amsthm} 
\usepackage{aliascnt} 

\numberwithin{equation}{section}

\usepackage{amsmath,amssymb,mathrsfs,mathtools,braket,latexsym,esint}

\newcommand{\scrF}{\mathscr{F}}
\newcommand{\scrG}{\mathscr{G}}
\newcommand{\scrM}{\mathscr{M}}
\newcommand{\NN}{\mathbb{N}}
\newcommand{\RR}{\mathbb{R}}
\newcommand{\eps}{\varepsilon}
\renewcommand{\epsilon}{\eps}
\renewcommand{\phi}{\varphi}
\renewcommand{\theta}{\vartheta}
\newcommand{\mathd}{\mathrm{d}}
\newcommand{\LN}{\mathscr{L}^N}
\newcommand{\Hnmu}{\mathscr{H}^{N-1}}
\newcommand{\Snmu}{\mathbb{S}^{N-1}}
\renewcommand{\hat}{\widehat}
\newcommand{\mz}{\setminus\{0\}}
\newcommand{\RNmz}{\RR^N\setminus \{0\}}

\usepackage{amsrefs}
\usepackage[shortlabels]{enumitem}

\usepackage[dvipsnames]{xcolor}

\newcommand{\EEE}{\color{magenta}}

\newcommand{\BBB}{\color{black}}
\usepackage{comment}
\usepackage[
colorlinks=true, pdfstartview=FitV, linkcolor=blue, 
            citecolor=cyan, urlcolor=cyan]{hyperref}

\theoremstyle{plain}
\newtheorem{theorem}{Theorem}[section]
	
\newaliascnt{proposition}{theorem}
\newtheorem{proposition}[proposition]{Proposition}
\aliascntresetthe{proposition}

\newaliascnt{corollary}{theorem}
\newtheorem{corollary}[corollary]{Corollary}
\aliascntresetthe{corollary}

\newaliascnt{lemma}{theorem}
\newtheorem{lemma}[lemma]{Lemma}
\aliascntresetthe{lemma}

\newaliascnt{definition}{theorem}
\theoremstyle{definition}

\aliascntresetthe{definition}

\newaliascnt{remark}{theorem}
\newtheorem{remark}[remark]{Remark}
\aliascntresetthe{remark}

\newaliascnt{example}{theorem}
\newtheorem{example}[example]{Example}
\aliascntresetthe{example}

\title[Sharp conditions for the BBM formula]{Sharp~conditions~for~the~validity of~the~Bourgain-Brezis-Mironescu~formula}
\date{\today}

\author[E. Davoli]{Elisa Davoli}
\author[G. Di Fratta]{Giovanni Di Fratta}
\author[V. Pagliari]{Valerio Pagliari}	

\AtEndDocument{
	\bigskip
	\footnotesize{
		\noindent
		E. Davoli,
		
		\noindent
		Institute of Analysis and Scientific Computing,
		TU Wien,
		
		\noindent
		Wiedner Hauptstra\ss e 8-10, 1040 Vienna, Austria.
		
		\noindent
		E-mail:
		\href{mailto:elisa.davoli@tuwien.ac.at}{\tt elisa.davoli@tuwien.ac.at}.
	}
	
	\medskip
	\footnotesize{
		\noindent
		G. Di Fratta,
		
		\noindent
		Dipartimento di Matematica e Applicazioni ``R. Caccioppoli'',
		Universi\`a degli studi di Napoli ``Federico II'',
		
		\noindent
		Via Cintia, Complesso Monte S. Angelo, 80126 Naples, Italy.
		
		\noindent
		E-mail:
		\href{mailto:giovanni.difratta@unina.it}{\tt giovanni.difratta@unina.it}.
	}
	
	\medskip
	\footnotesize{
		\noindent
		V. Pagliari,
		
		\noindent
		Institute of Analysis and Scientific Computing,
		TU Wien,
		
		\noindent
		Wiedner Hauptstra\ss e 8-10, 1040 Vienna, Austria.
		
		\noindent
		E-mail:
		\href{mailto:valerio.pagliari@tuwien.ac.at}{\tt valerio.pagliari@tuwien.ac.at}.
	}
}

\begin{document}
	\begin{abstract}
		Following the seminal paper by Bourgain, Brezis and Mironescu,
        we focus on the asymptotic behavior of some nonlocal functionals
        that, for each $u\in L^2(\RR^N)$, are defined
        as the double integrals of weighted, squared difference quotients of $u$.
        Given a family of weights $\{\rho_\eps\}$, $\eps\in(0,1)$,
        we devise sufficient and necessary conditions on $\{\rho_\eps\}$
        for the associated nonlocal functionals to converge as $\eps \to 0$
        to a variant of the Dirichlet integral.
        Finally, some comparison between our result and the existing literature is provided.    
		
		
	\end{abstract}
	
	\maketitle

\section{Introduction}\label{sec:intro}
Let $J \coloneqq (0, 1)$ and let $u\colon \RR^N \to \RR$ be an $L^2$ function.
Given the family of kernels $\{\rho_\eps\}_{\eps\in J}$,
with $\rho_\eps\colon \RR^N \to [0,+\infty)$ measurable,
we consider the energy functionals
	\begin{equation}\label{eq:energy}
		\scrF_\eps [ u ] \coloneqq
		\frac{1}{2}\int_{\RR^N \times \RR^N}
			\rho_\eps (y-x)
			\frac{| u (y) - u (x) |^2}{| y - x |^2} 
		\mathd y \mathd x.
\end{equation}
We aim at characterizing the class of kernels such that
for every $u\in H^1(\RR^N)$ the family
$\{\scrF_\eps [ u ]\}$ converges to (a variant of) $\| \nabla u \|^2_{L^2(\RR^N)}$
as $\eps\to0$,
see~\autoref{thm:main}.

Our study follows the line of research initiated
 in the renowned paper \cite{BBM01}.
The motivation advanced by the authors was the analysis of the the Gagliardo seminorms
\[
    [u]^p_s \coloneqq \int_{\RR^N \times \RR^N}
			\frac{| u (y) - u (x) |^2}{| y - x |^{N+sp}} 
		\mathd y \mathd x,
	\qquad \text{with } p\in(1,+\infty),s\in(0,1)
\]
as $s\to 1$.
They studied the asymptotics as $\eps \to 0$ of double integrals with the same structure as the ones in \eqref{eq:energy} for a family $\{\rho_\eps\}\subset L^1(\RR^N)$ of radial kernels
and a general exponent $p\in(1,+\infty)$,
and they proved that the Sobolev seminorm $\|\nabla u\|^p_{L^p(\RR^N)}$ is retrieved in the limit.
The case of the Gagliardo seminorms may be treated analogously,
upon taking some extra care of the tails of the fractional kernel
(see, e.g.,~\cite[Sec. 1]{LS11}).

The literature on nonlocal-to-local formulas has become extremely vast, and
a detailed overview is beyond the scope of our contribution. 
Here, we restrict ourselves
to the research that is most close in spirit to \cite{BBM01}.
The gap left open for the case $p=1$ was filled in \cite{Dav02}, 
where a characterization of functions of bounded variation was provided
(see also~\cites{Pon04,LS11}).
The case of vector fields of bounded deformations was later addressed in \cite{Men12}
by considering a suitable symmetrization
of the functionals in \eqref{eq:energy}
(see also \cite{MQ15} for the asymptotics of nonlocal elastic energies of peridynamic-type and
\cite{SM19} for a study of fractional Korn inequalities).
The analysis of the asymptotic behavior
in the sense of $\Gamma$-convergence \cite{DM93}
of the fractional perimeter functionals introduced in \cite{CRS10} 
was undertaken in \cite{ADM11}, and
then extended in multiple directions
by several contributions, 
e.g.~\cite{BP19,Pag20,DKP22,MRT19}.
Finally, we point out that
a general variational framework for the analysis
of (static and dynamic) multiscale problems
that feature nonlocal interactions
has been very recently considered in the monograph~\cites{AAB22},
again for kernels that,
in our notation, are required to form a definitively bounded sequence in $L^1$.

A common trait of the works above is that
they only concern {\em sufficient conditions} for the nonlocal-to-local formulas to hold.
In the specific case of the functionals in~\eqref{eq:energy}
(see \autoref{thm:ponce} below for a prototypical statement),
this means that,
given a measurable map $\rho_\eps\colon \RR^N \to [0,+\infty)$ for every $\eps\in J$,
a set of conditions on the family $\{\rho_\eps\}_{\eps\in J}$ is prescribed,
so that the following can be deduced:
there exist	an infinitesimal sequence $\{\eps_k\}\subset J$ and
a positive Radon measure $\lambda$ on the unit sphere $\Snmu$
that depends only on $\{\rho_{\eps}\}$ such that
for every $u\in H^1(\RR^N)$
	\begin{equation}\label{eq:BBM}
			\lim_{k\to+\infty} \scrF_{\eps_k}[u]
			= \int_{\RR^N} \int_{\Snmu} | \nabla u(x) \cdot \sigma |^2 \mathd \lambda(\sigma) \mathd x.	
   \end{equation}
We refer to such equality as the
{\em Bourgain-Brezis-Mironescu formula},
in short BBM formula.
The novelty of our contribution is that
{\em we devise conditions that are both necessary and sufficient}
for \eqref{eq:BBM} to hold
(see also \autoref{subsec:comments} for some remarks about energies with non-quadratic growth).
Precisely, we establish the following.

\begin{theorem}[Necessary conditions for the BBM formula]\label{thm:main}
    For every $\eps\in J$,
	let $\rho_\eps\colon \RR^N \to [0,+\infty)$ be measurable and
	let $\scrF_\eps$ be as in~\eqref{eq:energy}.
	Let also $\lambda$ be a fixed positive Radon  measure
	on the unit sphere $\Snmu$.
	
	Suppose that
	there exists an infinitesimal sequence $\{\eps_k\}\subset J$
	such that for every $u\in H^1(\RR^N)$
	the Bourgain-Brezis-Mironescu formula \eqref{eq:BBM} holds for the given measure $\lambda$.
	Then,
	the sequence $\{\rho_{\eps_k}\}$ satisfies the following:
	\begin{enumerate}[label=\emph{(\roman*)}]
	\item\label{i} there exists $M\geq0$ with the property that
	    	for every $R>0$
            \begin{gather}\label{eq:nec}
            \limsup_{k\to+\infty}\left[
            \int_{B(0,R)} \rho_{\eps_k}(z)\mathd z
            +
            R^2 \int_{B(0,R)^c} \frac{\rho_{\eps_k}( z)}{| z |^2} \mathd z
            \right]
            \leq M;
            \end{gather}
	 \item\label{ii} the sequence $\{\nu_k\}$ of Radon measures on $\RR^N$ defined by
        \begin{equation}\label{eq:nuk}
	 		\langle \nu_k , f \rangle
	 			\coloneqq  \int_{\RR^N} \rho_{\eps_k}(z) f(z) \mathd z
	 		\qquad\text{for all } f\in C_c(\RR^N).
        \end{equation}
	 	locally weakly-$\ast$ converges in the sense of Radon measures
        to $\alpha \delta_0$,
        where $\alpha\geq 0$ is a positive constant,
        and $\delta_0$ is the Dirac delta in $0$.
	\end{enumerate}
\end{theorem}

Roughly speaking, condition \ref{i} prescribes that
for $\eps\in J$ small enough
each kernel $\rho_\eps$ must have finite mass in any large ball around the origin, and that, at the same time,
the contributions accounting for long-range interactions must be asymptotically negligible.
Indeed, as we show in \autoref{subs:cond1},
\eqref{eq:nec} is equivalent to the following \emph{uniform decay condition}:
there exists $M\geq 0$ such that for every $R>0$
\begin{gather*}
    \limsup_{k\to +\infty}\int_{\RR^N}\frac{\rho_{\eps_k}(z)}{R^2+|z|^2}\mathd z
    \leq \frac{M}{R^2}.
\end{gather*}

When $R=1$,
the  previous inequality entails that
for $k$ large enough $\rho_{\eps_k}\in L^1_{\rm loc}(\RR^N)$, so that, in particular,
position \eqref{eq:nuk} actually defines a Radon measure on $\RR^N$.
A useful way
to regard the measures $\nu_k$ in \eqref{eq:nuk}
is to think of them as quantities encoding medium-range interactions,
although this is not immediately evident from the definition.
From this point of view,
condition \ref{ii} tells us that, in the limit,
such interactions must vanish outside of the origin.
We will elaborate further on this point in this introduction.

It turns out that
conditions \ref{i} and \ref{ii} are also sufficient
or the BBM formula to hold,
so that, in light of \autoref{thm:main},
they are sharp.
To establish the sufficiency,
we  need the following compactness result,
which is interesting on its own:

\begin{theorem}[Asymptotic behavior of nonlocal energies]\label{thm:compactness}
	For every $\eps\in J$,
	let $\rho_\eps\colon \RR^N \to [0,+\infty)$ be measurable and
	let $\scrF_\eps$ be as in~\eqref{eq:energy}.
    
    Suppose that
    there exists $M\geq0$ with the property that
	for every $R>0$
        \begin{gather}\label{eq:suff}
	    \limsup_{\eps\to 0}
        \left[
        \int_{B(0,R)} \rho_\eps(z)\mathd z
        +
        R^2 \int_{B(0,R)^c} \frac{\rho_\eps( z)}{| z |^2} \mathd z
        \right]
         \leq M.
        \end{gather}
    Then, there exist an infinitesimal sequence $\{\eps_k\}\subset J$ and
    two finite positive Radon measures $\mu$ and $\nu$,
    respectively on $\Snmu$ and $\RR^N$,
    that depend only on $\{\rho_{\eps_k}\}$, and
    such that for every $u\in H^1(\RR^N)$ 
    there holds 
    \begin{equation}\label{eq:cpt}
		\lim_{k\to+\infty} \scrF_{\eps_k}[u]
		=
		\frac{1}{2}\int_{\RR^N}
		\left[
		    \int_{\Snmu} | \nabla u(x) \cdot \sigma |^2 \mathd \mu(\sigma)
		    +
		    \int_{\RR^N\mz}
				\frac{| u (x+z) - u (x) |^2}{| z |^2} \mathd \nu(z)
		\right]
		\mathd x.
	\end{equation}
    Moreover, the right-hand side is finite for every $u\in H^1(\RR^N)$.
 \end{theorem}

\autoref{thm:compactness} shows that,
while the integrability and decay conditions in \ref{i} are sufficient to establish the convergence of the functionals in \eqref{eq:energy},
in the absence of condition \ref{ii}
we cannot exclude the persistence of nonlocal terms in the limit.
Indeed, the measure $\nu$ is retrieved as the limit (in the sense of weak-$\ast$ convergence)
of the medium-range interactions encoded by \eqref{eq:nuk}.
The measure $\mu$ captures instead the concentration of the sequence $\{\rho_{\eps_k}\}$ around the origin, and
it characterizes the (possibly zero) local term in the limiting energy.
Loosely speaking, for every Borel subset $E\subseteq \Snmu$, $\mu$ is given by
\begin{equation*}
    \mu(E)\coloneqq \lim_{\delta\to 0}\int_{C_\delta(E)}\rho_{\eps_\delta}(z) \mathd z
\end{equation*}
where $C_\delta(E)$ is the intersection of the cone spanned by $E$ with $B(0,\delta)$, $\{\eps_\delta\}$ is a suitable subfamily, and
the limit is taken in the sense of the weak-$\ast$ convergence of measures.
We refer to Step 3 and 4 in the proof of \autoref{stm:cpt} for the precise definition.
In particular, when the kernels $\rho_{\eps_k}$ are radial (cf.~\cite{BBM01}), then $\mu=c\mathcal{L}^N$ for a constant $c\geq0$.

We conclude our analysis by showing that,
when \ref{ii} is imposed as well,
the limiting nonlocal effects vanish.
\BBB

\begin{corollary}[Sharp sufficient conditions for the BBM formula]\label{cor:suff}
    Let us suppose that
    same hypotheses of \autoref{thm:compactness} hold,
    and let us suppose also that
    the family $\{\nu_\eps\}_{\eps\in J}$ of Radon measures on $\RR^N$ defined by
    \begin{equation}\label{eq:nueps}
	 	\langle \nu_\eps , f \rangle
	 		\coloneqq  \int_{\RR^N} \rho_{\eps}(z) f(z) \mathd z
	 	\qquad\text{for all } f\in C_c(\RR^N).
    \end{equation}
	locally weakly-$\ast$ converges in the sense of Radon measures
    to $\alpha \delta_0$,
    where $\alpha\geq 0$ is a positive constant,
    and $\delta_0$ is the Dirac delta in $0$.
    Then, there exist an infinitesimal sequence $\{\eps_k\}\subset J$ and
    a finite positive Radon measure $\mu$ on $\Snmu$ such that the Bourgain-Brezis-Mironescu formula holds,
	that is,
	\begin{equation}
 \label{eq:BBM-limit}
		\lim_{k\to+\infty} \scrF_{\eps_k}[u]
		=
		\frac{1}{2} \int_{\RR^N}\int_{\Snmu}
		| \nabla u(x) \cdot \sigma |^2
		\mathd \mu(\sigma) \mathd x.
	\end{equation}
\end{corollary}

 We refer to Remark \ref{rk:quad} below for an alternative formulation of the right-hand side of \eqref{eq:BBM-limit} in terms of the action of a quadratic form.

Our approach grounds on the use of the Fourier transform,
which allows recasting the family of nonlocal functionals in \eqref{eq:energy}
into double integrals of the form
\begin{equation}\label{eq:energia-F-intro}
    \int_{\RR^N} | \psi (\xi) |^2
		\int_{\RR^N} \rho_\eps( z)\frac{1 - \cos( z\cdot \xi )}{| z |^2}
	\mathd z \mathd \xi,
\end{equation}
with $\psi$ in a suitable weighted $L^2$ space
(see \eqref{eq:L2w} and \eqref{eq:energia-F}).
The technical preliminaries
about the Fourier transform and those on Radon measures to be used later in this work
are collected in \autoref{sec:pre}.
In particular,
the functionals in \eqref{eq:energia-F-intro} and an equivalent formulations of the BBM formula in Fourier variables are retrieved in \autoref{lemma:fourier}.

From \autoref{sec:cpt} we turn to the proof of our results.
First, we establish \autoref{thm:compactness}
by observing that
the condition in \eqref{eq:suff}
grants not only that
the integrals with respect to $z$ in \eqref{eq:energia-F-intro},
as function of~$\xi$,
grow at most as $1+|\xi|^2$
(see~\autoref{stm:growth}),
but also that they converge pointwise
to the Fourier transform of the integrals
within the square brackets in \eqref{eq:cpt}
(see~\autoref{stm:cpt}).
The dominated convergence theorem then applies,
and \eqref{eq:cpt} is retrieved.

The pointwise convergence of the nonlocal energies
provided by \autoref{stm:cpt} plays a central role in our analysis.
It is obtained by studying separately the behaviors
of the family $\{\rho_\eps\}$ at three distinct interaction ranges,
respectively short, medium and long,
that we encode by means of an additional parameter $\delta\in J$.
Short-range interactions arise from the contributions
of shrinking balls of radius $\delta$ centered in the origin, and,
as $\delta \to 0$,
they asymptotically approach the gradient term in \eqref{eq:cpt}.
Medium-range interactions originate from the contributions to the energy
stored in annuli that lie at a distance $\delta$ from the origin.
In the limit, their presence leads to the nonlocal term in \eqref{eq:cpt}, that is,
the integral with respect to the measure $\nu$.
Finally,
long-range interactions occur outside of balls of radius $\delta^{-1}$ centered in the origin,
and their contributions is negligible when $\delta\to 0$.

The proofs of our two other results are provided in \autoref{sec:suffnec}.
With \autoref{thm:compactness} on hand,
\autoref{cor:suff}, that is,
the sufficiency
of conditions \ref{i} and \ref{ii} in \autoref{thm:main}
for the BBM formula, 
follows quickly:
it is enough to observe that
\ref{ii} forces the integral with respect to $\nu$ in \eqref{eq:cpt} to vanish.
In this sense, \ref{ii} may be regarded
as a {\em locality condition},
since it requires that
in the limit the kernels concentrate in the origin.
Conditions of this sort appear to be natural
as far as sufficient criteria
for the convergence of the nonlocal energies to variants of the Dirichlet norm
are sought after
(cf.,~e.g.,~\eqref{eq:concentr} in \autoref{thm:ponce} below or \cite[Thm. 3.1]{AAB22}).
The key novelty of our contribution is that 
we prove item \ref{ii} in \autoref{thm:main}
to be the weakest locality requirement
for the BBM formula \eqref{eq:BBM} to hold.

Proving \autoref{thm:main}, that is, the necessity of \ref{i} and \ref{ii} for the validity of the BBM formula,
is a more delicate issue.
The key step is established in \autoref{stm:concentration}, where,
by a suitable scaling of the functions in \eqref{eq:energia-F-intro}
(see~\autoref{rmk:poincare}),
it is proved that \eqref{eq:ponce} implies \ref{i}.
The weak-$\ast$ convergence of the sequence $\{\nu_k\}$ in \ref{ii} to a multiple of the Dirac delta in $0$ follows then from a homogeneity argument.
We conclude our contribution in \autoref{sec:final}
by clarifying how it compares with the existing literature and
by pointing out possible future research directions.

As we briefly outlined above,
there have been intense research efforts in the asymptotic analysis of nonlocal energies of the form \eqref{eq:energy}.
It is to be noted that
such functionals also arise in applications,
a case of interest being represented, for instance, by nonlocal models in micromagnetics.
Indeed, as pointed out in \cite{Rogers},
if the classical symmetric exchange energy
given by the Dirichlet integral of the magnetization
is replaced by a nonlocal Heisenberg functional of the form \eqref{eq:energy},
then a model closer to atomistic theories is obtained, and, in addition,
the class of admissible magnetizations may be enlarged
to include discontinuous and even `measure-valued' fields. 
This observation is crucial in nonconvex problems such as those of ferromagnetism,
in which the highly oscillatory `domain structures' observed in ferromagnetic materials cannot be captured by magnetizations with Sobolev regularity.
In such nonlocal micromagnetics models, knowing what classes of kernels $\rho_\eps$ lead to an approximation of the classical Dirichlet energies amounts to a selection criterion to establish whether nonlocal descriptions can be replaced by local ones or, instead, such approximations are not mathematically correct. We refer to \cite{DFS22} for further discussion on this topic.


\section{Preliminaries}\label{sec:pre}
After fixing the notation,
in this section we provide a concise overview of some facts
from the theories of the Fourier transform and of Radon measures,
which will serve as the main tools for our study.
In particular,
in \autoref{lemma:fourier}
we derive an equivalent form of the BBM formula \eqref{eq:BBM}
to be employed as the cornerstone of our analysis.

For $N\in \NN\setminus \{0\}$,
we work in the $N$-dimensional Euclidean space $\RR^N$,
endowed with the corresponding inner product $\,\cdot\,$ and norm $|\;|$.
We let $\{e_1,\dots,e_N\}$ be its canonical basis.
For all $z\in \RNmz$ we define $\hat z \coloneqq z / | z |$.
We denote by $\LN$ and $\Hnmu$
the $N$-dimensional Lebesgue and
the $(N-1)$-dimensional Hausdorff measures, respectively.
We let $B(x,r)$ be the open ball in $\RR^N$ of center $x$ and radius~$r$.
We write $B(x,r)^c$ for the complement of $B(x,r)$, while 
the topological boundary of $B(0,1)$ is denoted by $\Snmu$.

\subsection{Fourier transform}
In this paper, we resort
to results on the Fourier transform that are standard and
can be found in any textbook on Fourier analysis
(see, e.g.,~\cite{Stein}).
Here we briefly recall the properties to be used below.

We will employ the unitary Fourier transform
expressed in terms of angular frequency, that is,
for any rapidly decaying $u\in C^\infty(\RR^N)$ and $\xi \in \RR^N$
\[
	\mathcal{F}u(\xi)\coloneqq
		\frac{1}{(2\pi)^{N/2}} \int_{\RR^N} e^{-{\rm i} x\cdot \xi} u(x)\mathd x.
\]
As customary, we will adopt $\hat{u}$ as a shorthand for $\mathcal{F}u$.
We recall that the following identities hold:
\begin{gather*}
	\hat{\tau_z u}(\xi)=e^{-{\rm i}z\cdot \xi} \hat{u}(\xi),
	\qquad
	\hat{\partial_\alpha u}(\xi)=({\rm i} \xi )^\alpha \hat{u}(\xi),
\end{gather*}
where $( \tau_z u ) (x) \coloneqq u(x - z)$, for $x,z,\xi\in \RR^N$,
and where $\alpha\in\NN^N$ is a multi-index.
In particular, we observe that, by the Parseval identity,
the Fourier transform is a bijection between
\[
	H^1(\RR^N)\coloneqq \left\{ u\in L^2(\RR^N) : \text{the distribution $\nabla u$ is in } L^2(\RR^N)\right\}
\]
and the weighted space
\begin{equation}\label{eq:L2w}
    L^2_w(\RR^N) \coloneqq \left\{ \psi\in L^2(\RR^N) :
		\int_{\RR^N} |\xi|^2|\psi(\xi)|^2\mathd \xi <+\infty \right\}.
\end{equation}

By applying Fourier techniques to the functionals in \eqref{eq:BBM},
the following is readily obtained.

\begin{lemma}\label{lemma:fourier}
    Let $\lambda$ be a positive Radon measure on $\Snmu$.
    For every $u \in H^1 \left( \RR^N \right)$ we define
    \begin{equation}\label{eq:limit-func}
        \scrF[u]\coloneqq \int_{\RR^N} \int_{\Snmu} | \nabla u(x) \cdot \sigma |^2 \mathd \lambda(\sigma) \mathd x,
    \end{equation}
    while for every $\psi \in L^2_w(\RR^N)$ we set
    \begin{gather}
	\hat\scrF_\eps[\psi]\coloneqq \int_{\RR^N} | \psi (\xi) |^2
							\int_{\RR^N} \rho_\eps( z)\frac{1 - \cos( z\cdot \xi )}{| z |^2}
						\mathd z \mathd \xi, 
	\label{eq:energia-F}\\
	\hat\scrF[\psi]\coloneqq \int_{\RR^N} |\psi (\xi) |^2
							\int_{\Snmu} | \xi \cdot \sigma |^2 \mathd \lambda(\sigma) 
							\mathd \xi.
	\label{eq:limit-F}
    \end{gather}
    Then, recalling \eqref{eq:energy},
    for every $u \in H^1 \left( \RR^N \right)$ it holds
    \[
        \scrF_\eps[u]=\hat\scrF_\eps[\hat u],
        \qquad
        \scrF[u]=\hat\scrF[\hat u],
    \]
	and, in particular,
    there exist	an infinitesimal sequence $\{\eps_k\}\subset J$ such that \eqref{eq:BBM} holds
    for every $u\in H^1(\RR^N)$
    if and only if for every $\psi \in L^2_w(\RR^N)$
    \begin{equation}\label{eq:nonloc-to-loc-F}
       \lim_{k\to+\infty}\hat\scrF_{\eps_k}[\psi]=\hat\scrF[\psi].
    \end{equation}
\end{lemma}
\begin{proof}
	Recall that $( \tau_z u ) (x) \coloneqq u(x - z)$ for every $x,z\in \RR^N$.
	By the change of variables $z \coloneqq y - x$ and the Parseval identity
	we obtain
	\begin{align*}
	\scrF_\eps [ u ]
		& = \frac{1}{2} \int_{\RR^N \times \RR^N}
				\frac{\rho_\eps(z)}{| z |^2} | u (x + z) - u (x) |^2  
			\mathd z \mathd x \\
		& = \frac{1}{2} \int_{\RR^N}
				\frac{\rho_\eps(z)}{| z |^2}
				\int_{\RR^N} | \tau_{- z} u (x) - u (x) |^2
			\mathd x \mathd z \\
		& = \frac{1}{2}	\int_{\RR^N}
				\frac{\rho_\eps( z)}{| z |^2}
				\int_{\RR^N} | \mathcal{F} [ u - \tau_{- z}u ] (\xi) |^2
			\mathd \xi \mathd z.
	\end{align*}
	The properties of the Fourier transform yield
	\[
		| \mathcal{F} [ u - \tau_{- z}u ] (\xi) |^2
		= | 1 - e^{{\rm i}  z\cdot \xi} |^2 | \widehat{u} (\xi) |^2
		= 2 \big( 1 - \cos( z\cdot \xi ) \big) | \widehat{u} (\xi) |^2,
	\]
	whence we infer $\scrF_\eps[u]=\hat\scrF_\eps[\hat u]$.
    Similarly, we have
	\begin{align*}
		\scrF[u]
		& = \int_{\Snmu} \int_{\RR^N} | \mathcal{F}[\nabla u \cdot \sigma](\xi) |^2 \mathd \xi \mathd \lambda(\sigma) \\
		& = \int_{\RR^N} |\widehat{u} (\xi) |^2 \int_{\Snmu} | \xi \cdot \sigma |^2 \mathd \lambda(\sigma)  \mathd \xi \\
        & = \hat\scrF[\hat u].
	\end{align*}
	We then achieve the conclusion
    thanks to the one-to-one correspondence
	between $H^1(\RR^N)$ and $L^2_w(\RR^N)$
	provided by the Fourier transform.
\end{proof}

\subsection{Positive Radon measures on $\RR^N$}\label{sec:radon}
We recall here some definitions and properties
that may be found, e.g., in \cite[Secs. 1.3 and 1.4]{AFP00};
we refer to such monograph for a more detailed study of (geometric) measure theory.
 
Let $X \subseteq \RR^N$ be a set.
A positive measure $\mu$
on the $\sigma$-algebra of Borel sets in $X$
is a positive Radon measure
if it is finite on compact sets;
if it holds as well that $\mu(X)<+\infty$,
we say that $\mu$ is a finite positive Radon measure.
We denote the space of positive Radon measures on $X$ by $\scrM_{\rm loc}(X)$
and the one of finite positive Radon measures by $\scrM(X)$.

The Riesz representation theorem proves that
$\scrM_{\rm loc}(X)$ may be identified as the dual
of the space of compactly supported continuous functions $C_c(X)$
endowed with local uniform convergence.
Accordingly, we say that
a sequence $\{\mu_k\}\subset\scrM_{\rm loc}(X)$ converges
to $\mu\in \scrM_{\rm loc}(X)$ in the local weak-$\ast$ sense,
and we write $\mu_k\stackrel{\ast}{\rightharpoonup} \mu$ in $\scrM_{\rm loc}(X)$,
if
\begin{equation}
\label{eq:cc}
	\lim_{k\to+\infty} \int_{X} f(x)\mathd\mu_k(x) =
	\int_{X} f(x)\mathd\mu(x)
	\qquad\text{for every } f\in C_c(X).
\end{equation}
In wider generality,
if $\mu_k\stackrel{\ast}{\rightharpoonup} \mu$ in $\scrM_{\rm loc}(X)$,
then the previous equality holds
for every bounded Borel function $f\colon X \to \RR$ with compact support
such that the set of its discontinuity points is $\mu$-negligible.
In particular, if $X$ is compact and $\mu_k\stackrel{\ast}{\rightharpoonup} \mu$ in $\scrM_{\rm loc}(X)$, then \eqref{eq:cc} holds for every $f\in C(X)$.

A uniform control on the mass of each compact set
along a sequence of Radon measure is sufficient to ensure
local weak-$\ast$ precompactness:
if $\{\mu_k\}$ is a sequence of positive Radon measures
such that $\sup_k\{ \mu_k(C) : C\subset X\} <+\infty$
for every compact set $C \subset X$,
then there exists a locally weakly-$\ast$ converging subsequence.

\section{Proof of \autoref{thm:compactness}}\label{sec:cpt}

We devote this section to proving that
the summability and decay conditions in \eqref{eq:suff} are sufficient
to yield convergence of a subsequence of $\{\scrF_\eps\}$.
In particular, we are able to characterize the limiting functional,
as \eqref{eq:cpt} shows.

As a first step, by assuming that
the kernels $\rho_{\eps}$ satisfy \eqref{eq:suff}
(actually, it suffices that the bound holds
just for one $R>0$),
we deduce that
the energies $\hat{\scrF}_\eps$ in \eqref{eq:energia-F} are finite for every $\psi\in L^2_w(\RR^N)$,
provided $\eps$ is small enough.
This is an immediate consequence of the next lemma, 
which, in spite of its simplicity, will prove to be useful.
\begin{lemma}\label{stm:growth}
    For every $\eps\in J$,
	let $\rho_\eps\colon \RR^N \to [0,+\infty)$ be measurable, and
    let us suppose that
    \eqref{eq:suff} holds for $R=1$.
    Then, for every $\xi\in \RR^N$
    \begin{gather*}
        \limsup_{\eps \to 0}
        \int_{B(0,1)}
            \rho_\eps( z)\frac{ 1 - \cos(z\cdot \xi )}{| z |^2}
        \mathd z
        \leq \frac{M}{2} |\xi|^2, \\
        \limsup_{\eps \to 0}
        \int_{B(0,1)^c}
            \rho_\eps( z)\frac{ 1 - \cos(z\cdot \xi )}{| z |^2}
        \mathd z
        \leq 2M,
	\end{gather*}
    where $M\geq 0$ is as in \eqref{eq:suff}.
\end{lemma}
\begin{proof}
    From \eqref{eq:suff} with $R=1$,
    it follows
    \begin{gather}\label{eq:levy}
    \limsup_{\eps \to 0}
        \int_{B(0,1)} \rho_\eps(z)\mathd z
    \leq M,
    \qquad
    \limsup_{\eps \to 0}
        \int_{B(0,1)^c} \frac{\rho_\eps( z)}{| z |^2} \mathd z
    \leq M
    \end{gather}
    
    We first focus on contributions in $B(0,1)$.
	Since $\sin (t) \leq t$ for $t \geq 0$, we have
	\begin{align}\label{eq:bound-cos}
	\frac{1 - \cos ( z \cdot \xi)}{| z |^2}
				= \frac{1}{|z|^2}\int_0^{| z \cdot \xi|} \sin(t) \mathd t
				\leq \frac{1}{2} (\hat z\cdot \xi)^2,
	\end{align}
	where $\hat{z}\coloneqq z/|z|$.
    By taking into account the first inequality in \eqref{eq:levy},
    we deduce
	\begin{equation*}
	    \limsup_{\eps \to 0}
		\int_{B(0,1)} \rho_\eps( z) \frac{1 - \cos(z\cdot \xi )}{| z |^2} \mathd z
		\leq \frac{| \xi |^2}{2}
		\limsup_{\eps \to 0}
		\int_{B(0,1)} \rho_\eps( z)  \mathd z
		\leq \frac{M}{2} | \xi |^2.
	\end{equation*}
    Instead, far from the origin we have
	\begin{equation*}
		\limsup_{\eps \to 0}
		\int_{B(0,1)^c} \rho_\eps( z) \frac{1 - \cos(z\cdot \xi )}{| z |^2} \mathd z
		\leq 2 \limsup_{\eps \to 0}
		\int_{B(0,1)^c} \frac{\rho_\eps( z)}{| z |^2} \mathd z
		\leq 2M,
	\end{equation*}
	where we used the second estimate in \eqref{eq:levy}.	
\end{proof}

For the second step
towards the proof of \autoref{thm:compactness},
it is convenient to introduce the following notation:
for every $\xi\in\RR^N$ and $\eps\in J$, we let
\begin{equation}\label{eq:Ieps}
	I_\eps(\xi;A)\coloneqq 
		\int_{A} \rho_\eps(z)\frac{1-\cos( z\cdot \xi)}{| z|^2}
	\mathd z.
	\qquad\text{for all $\LN$-measurable } A\subseteq \RR^N.
\end{equation}
By \autoref{stm:growth},
we know that, under condition \eqref{eq:suff}, the functional $I_\eps(\xi;\RR^N)$ grows at most as $1+|\xi|^2$.
Then,
recalling the formulation of the BBM formula in Fourier variables provided by \autoref{lemma:fourier},
in order to show that \eqref{eq:cpt} holds,
it suffices to characterize the pointwise limit of
the family of integrals with respect to $z$ in \eqref{eq:energia-F},
when regarded as functions of $\xi$, that is, of
$\{I_\eps(\,\cdot\,;\RR^N)\}$.
The next proposition takes care of this.

Note that in order to achieve the task
that we have just outlined
it is natural to regard $\{\rho_\eps\}$ as a family of Radon measures and
to take the limit of $\{I_\eps(\,\cdot\,;\RR^N)\}$
by appealing to some weak-$\ast$ compactness argument.
Even though such compactness is actually available
(see Step 2 in the proof of \autoref{stm:cpt}),
the discontinuity of the function $z \mapsto (1-\cos(\xi\cdot z))/|z|^2$ prevents the results recalled in \autoref{sec:radon}
from being immediately viable.
To circumvent such an obstacle,
in the proof of \autoref{stm:cpt}
we introduce an auxiliary parameter $\delta\in J$
to quantify the range of interactions
(respectively short, medium or long), and
we accordingly define two families of measures,
which are meant to encode the limiting behavior of $\{\rho_\eps\}$ at different scales.

\begin{proposition}\label{stm:cpt}
	If \eqref{eq:suff} holds,
	then there exist an infinitesimal sequence $\{\eps_k\}\subset J$ and
    two finite Radon measures $\mu\in \scrM(\Snmu)$ and $\nu\in\scrM(\RR^N)$
    that depend only on $\{\rho_{\eps_k}\}$
    and such that for every $\xi\in \RR^N$
	\begin{equation*}
		\lim_{k\to +\infty} I_{\eps_k}(\xi;\RR^N)=
		\frac{1}{2}
			\int_{\Snmu} |\xi \cdot \sigma|^2\mathd\mu(\sigma)
			+
			 \int_{\RNmz}
				\frac{1-\cos( z\cdot \xi)}{|z|^2}
			\mathd\nu(z).
	\end{equation*}
\end{proposition}
\begin{proof}
    Let us fix $\delta\in J$.
    In order to compute the desired limit
	we part $\RR^N$ in three regions:
	$B(0,\delta)$, $A_\delta$ and $B(0,\delta^{-1})^c$,
	where $A_\delta\coloneqq \{ z \in \RR^N : \delta <|z|<\delta^{-1}\}$.
	The proof is then divided into several steps:
	for each given $\delta\in J$ (except a countable family of them, see Step 2 below)
	we take the limits as $\eps\to 0$ of 
	$I_\eps(\xi;B(0,\delta))$, $I_\eps(\xi;A_\delta)$, and $I_\eps(\xi;B(0,\delta^{-1})^c)$.
    For the analysis of the first two terms
    the starting point is the observation that
    \eqref{eq:suff} implies for every $R>0$ the existence of $\bar\eps_R\in J$
    such that
    \begin{equation}\label{eq:L1}
       \int_{B(0,R)} \rho_\eps(z)\mathd z \leq M+1
       \qquad\text{for every $\eps \in (0,\bar\eps_R)$}
    \end{equation}
    (cf. \eqref{eq:levy}).
	In the final step we conclude
	by summing up the three contributions and taking the limit as $\delta\to0$
	
	\noindent\underline{\sc Step 1: long range interactions}.
	The term $I_\eps(\xi;B(0,\delta^{-1})^c)$ is readily estimated
	by means of \eqref{eq:suff}:
	for every $\delta\in J$ we have
	\begin{equation}\label{eq:long}
	    \limsup_{k\to+\infty} I_{\eps_k}(\xi;B(0,\delta^{-1})^c)\leq 2M \delta^2.
	\end{equation}
    
	\noindent\underline{\sc Step 2: medium range interactions}.
	For all $\eps\in J$,
	let us define the measure $\nu_\eps\coloneqq \rho_\eps \mathscr{L}^N$ (cf.~\eqref{eq:nueps}).
    Let $\{R^{(n)}\}_{n\in\NN}$ be a strictly increasing sequence of strictly positive real numbers.
	It follows from \eqref{eq:L1} that
	for every $n\in \NN$ there exists $\eta^{(n)}\in J$ such that it holds
	\begin{gather*}
	\int_{B(0,R^{(n)})} \mathd \nu_{\eps}
		= \int_{B(0,R^{(n)})} \rho_{\eps}(z)\mathd z \leq M+1.
        \qquad\text{for every } \eps\in (0,\eta^{(n)}).
	\end{gather*}
    We can choose each $\eta^{(n)}$ so that
    $\{\eta^{(n)}\}$ is strictly decreasing.
	From the previous bound,
    for each $n\in \NN$ we deduce the existence of a finite positive Radon measure $\nu^{(n)}\in \scrM(B(0,R^{(n)}))$ and 
    of a sequence $\{\eps_k^{(n)}\}\subset (0,\eta^{(n)})$ such that
    $\nu_{\eps_k^{(n)}} \stackrel{\ast}{\rightharpoonup} \nu^{(n)}$
	weakly-$\ast$ in $\scrM(B(0,R^{(n)}))$.
    By grounding on this property,
    a diagonal argument yields the existence
	of a sequence $\{\eps_k\}\subset J$ and
	of a Radon measure $\nu$ on $\RR^N$
	such that
	$\nu_{\eps_k} \stackrel{\ast}{\rightharpoonup} \nu$
	locally weakly-$\ast$ in $\scrM_{\rm loc}(\RR^N)$.
	In particular, by the lower semicontinuity of the total variation
	with respect to the weak-$\ast$ convergence,
	since $M$ does not depend on $R$,
	we infer that $\nu$ is finite.
	
	We next resort to a known property of Radon measures:
	if $\{E_\delta\}_{\delta\in J}$ is a family
	of pairwise disjoint Borel sets in $\RR^N$ and
	if $\mu\in \scrM_{\rm loc}(\RR^N)$,
	then $\mu(E_\delta)>0$ for at most countably $\delta\in J$
	(see \cite[page~29]{AFP00}).
	By applying this property to the family $\{\partial A_\delta\}_{\delta\in J}$ and the measure $\nu$,
	we deduce that
	the set of discontinuty points of the function
	\[
	    \chi_\delta (z)\coloneqq
	    \begin{cases}
	        0 & \text{if }  z\notin A_\delta, \\
	        1 & \text{if }  z\in A_\delta
	    \end{cases}
	\]
	is $\nu$-negligible for all $\delta\in J$,
	but those in a certain countable subset $C\subset J$.
	As a consequence,
	since $\{\nu_{\eps_k}\}$ weakly-$\ast$ converges to $\nu$,
	the following equality holds for every $\delta\in J\setminus C$:
	\begin{align}
		\lim_{k\to +\infty} I_{\eps_k}(\xi;A_\delta) 
			& =
			\lim_{k\to +\infty}
				\int_{\RR^N} \chi_{\delta}(z)\frac{1-\cos( z\cdot \xi)}{|z|^2} \mathd \nu_{\eps_k}(z)
			\nonumber\\
			& = \int_{A_\delta} \frac{1-\cos( z\cdot \xi)}{|z|^2} \mathd \nu(z). \label{eq:medium}
	\end{align}
	
	\noindent\underline{\sc Step 3: short range interactions}.
	We adapt the approach of \cite[Subsec. 1.1]{Pon04}.
	For a fixed $\delta\in J$ and each $\eps\in J$
	we define the Radon measure $\mu_\eps^{(\delta)}$ on $\Snmu$ by setting
	\begin{equation*}
	     \mu_{\eps}^{(\delta)}(E) \coloneqq
	        \int_E \left(\int_0^\delta t^{N-1} \rho_\eps(t \sigma)\mathd t\right)
	    \mathd\Hnmu(\sigma)
	    \qquad\text{for all $\Hnmu$-measurable sets } E\subset \Snmu.
	\end{equation*}
	By means of the coarea formula  we deduce from \eqref{eq:L1} with $R=1$ that
	definitively $\mu_\eps^{(\delta)}(\Snmu)\leq M+1$.
	Thus, for all $\delta\in J$,
	there exists an infinitesimal sequence $\{\eps_k^{(\delta)}\}\subset J$ and
    a finite Radon measures $\mu^{(\delta)}\in \scrM(\Snmu)$ such that
    $\mu_{\eps_k^{(\delta)}}^{(\delta)}\stackrel{\ast}{\rightharpoonup} \mu^{(\delta)}$ weakly-$\ast$ in $\scrM(\Snmu)$ as $k\to+\infty$.
    Note that it holds $\mu^{(\delta)}(\Snmu)\leq M+1$ for every $\delta\in J$.
    
    Next, by a Taylor expansion of the cosine in $0$ we obtain
    \begin{align*}
    I_{\eps}(\xi;B(0,\delta)) & =
        \frac{1}{2}\int_{B(0,\delta)}  \rho_{\eps}(z)|\xi\cdot \hat{z}|^2 \mathd z
        + \int_{B(0,\delta)} \rho_{\eps}(z) O(|\xi|^3|z|) \mathd z \\
        & = \frac{1}{2}\int_{\Snmu}  |\xi\cdot \sigma|^2 \mathd \mu_{\eps}^{(\delta)}(\sigma)
        + \int_{B(0,\delta)} \rho_{\eps}(z) O(|\xi|^3|z|) \mathd z.
    \end{align*}
	Since  $\sigma\mapsto |\xi\cdot \sigma|^2$ is a continuous function on $\Snmu$,
	in view of the weak-$\ast$ convergence of $\{\mu_{\eps_k^{(\delta)}}^{(\delta)}\}$
	we can take the limit as $k\to+\infty$.
	Thus, for every $\delta\in J$, we find
	\begin{align}\label{eq:short}
    \limsup_{k\to+\infty} I_{\eps_k^{(\delta)}}(\xi;B(0,\delta)) =
        \frac{1}{2}\int_{\Snmu}  |\xi\cdot \sigma|^2 \mathd \mu^{(\delta)}(\sigma)
        + \limsup_{k\to+\infty} \int_{B(0,\delta)} \rho_{\eps_k^{(\delta)}}(z) O(|\xi|^3|z|) \mathd z.
    \end{align}
	
	\noindent\underline{\sc Step 4: limit as $\delta\to 0$}.
	In order to achieve the conclusion,
	we need to take the limit as $\delta\to 0$
	of the terms considered in Steps 1 -- 3.
	
	To this aim,
	let us consider 
	the sequence $\{\eps_k\}\subset J$ and the set $C\subset J$ given by Step 2.
	Let also $\{\delta_n\}_{n\in\NN}\subset J\setminus C$ be an infinitesimal sequence.
	We observe that for any $n\in \NN$,
	by reasoning as in Step 3,
	we can inductively extract a subsequence $\{\eps_k^{(n)}\}\subset \{\eps_k^{(n-1)}\} \subset \{\eps_k\}$ such that
	the sequence of measures $\mu_k^{(n)}\coloneqq \mu^{(\delta_n)}_{\eps_k^{(n)}}$
	weakly-$\ast$ converges in $\scrM(\Snmu)$ to some $\mu^{(\delta_n)}$.
	Step 3 yields as well the existence
	of an unrelabeled  subsequence $\{\delta_n\}$ and
	of a Radon measure $\mu\in \scrM(\Snmu)$ such that
	the sequence $\{\mu^{(\delta_n)}\}$ weakly-$\ast$ converges in $\scrM(\Snmu)$ to $\mu$.
	
	Let us now define the diagonal sequence $\{\tilde \eps_k\}$
	by setting $\tilde \eps_k\coloneqq \eps_k^{(k)}$ for every $k\in \NN$.
	Then, recalling \eqref{eq:L1},
	it follows from \eqref{eq:short} that
	\begin{equation}\label{eq:int-mu}
	   \lim_{n\to+\infty}\limsup_{k\to+\infty} I_{\tilde \eps_k}(\xi;B(0,\delta_n)) =
	   \frac{1}{2}\int_{\Snmu}  |\xi\cdot \sigma|^2 \mathd \mu(\sigma).
	\end{equation}
	We also note that
	by monotone convergence we can take the limit also in \eqref{eq:medium}:
	\begin{equation}\label{eq:int-nu}
	   \lim_{n\to+\infty}\lim_{k\to+\infty} I_{\tilde \eps_k}(\xi;A_{\delta_n}) =
	   \int_{\RR^N\mz} \frac{1-\cos( z\cdot \xi)}{|z|^2} \mathd \nu(z).
	\end{equation}
	Eventually,
	by collecting \eqref{eq:long}--\eqref{eq:int-nu}, we get
	\begin{align*}
		\lim_{k\to +\infty} I_{\tilde \eps_k}(\xi;\RR^N)
		 = \lim_{n\to +\infty} \limsup_{k\to +\infty}
			\Big[I_{\tilde \eps_k}\big(\xi;B(0,\delta_n)\big)
			+I_{\tilde \eps_k}(\xi;A_{\delta_n})
			+I_{\tilde \eps_k}\big(\xi;B(0,\delta_n^{-1})^c\big)
			\Big],
	\end{align*}
	from which the conclusion follows.
\end{proof}

We are now in a position to prove \autoref{thm:compactness}.

\begin{proof}[Proof of \autoref{thm:compactness}]
    By \autoref{stm:growth}	we know that
	for $k$ sufficiently large
	$I_{\eps_k}(\xi;\RR^N)$ grows at most as $1+|\xi|^2$.
	\autoref{stm:cpt}, instead, characterizes the pointwise limit
	$\{I_{\eps_k}(\,\cdot\,;\RR^N)\}$,
	where $\{\eps_k\}\subset J$ is a suitable infinitesimal sequence.
	Thus, for every $\psi\in L^2_w(\RR^N)$,
	by dominated convergence we deduce 
	\[
		\lim_{k\to+\infty} \hat{\scrF}_{\eps_k}(\psi)=
		\int_{\RR^N}|\psi(\xi)|^2
		    \left[
		      \frac{1}{2}
			\int_{\Snmu} |\xi \cdot \sigma|^2\mathd\mu(\sigma)
			+
			\int_{\RNmz}
				\frac{1-\cos( z\cdot \xi)}{|z|^2}
			\mathd\nu(z)
		    \right]
		\mathd \xi
	\]
	where $\mu\in\scrM(\Snmu)$ and $\nu\in\scrM(\RR^N)$ are as in \autoref{stm:cpt}.
	Formula \eqref{eq:cpt} is then achieved
	by recalling that the Fourier transform is a one-to-one correspondence between $H^1(\RR^N)$ and $L^2_w(\RR^N)$, and
	by computations similar to the ones
	in the proof of \autoref{lemma:fourier}.
	
	We are now only left to show that
	the right-hand side in \eqref{eq:cpt} is finite for every $u\in H^1(\RR^N)$.
	As for the gradient term, its finiteness is trivial.
	For what concerns the nonlocal term, we note that
	in view of \autoref{stm:growth} and of the construction in \autoref{stm:cpt} there holds
    $$
    \int_{\RR^N\mz}\frac{1-\cos(z\cdot\xi)}{|z|^2}\mathd \nu(z)\leq 2M(1+|\xi|^2)
    $$
    pointwise in $\RR^N$.
    Thus, we deduce
    $$
    \int_{\RR^N} |\psi(\xi)|^2
        \int_{\RR^N\mz}\frac{1-\cos(z\cdot\xi)}{|z|^2}\mathd \nu(z)\mathd\xi
        \leq 2M \int_{\RR^N}|\psi(\xi)|^2(1+|\xi|^2)\mathd\xi$$
    for every $\psi\in L^2_w(\RR^N)$.
    The claim follows then
    by the same computations as in \autoref{lemma:fourier}.
\end{proof}

\section{Necessary and sufficient conditions for the BBM formula}\label{sec:suffnec}

The goal of this section is to prove that
conditions \ref{i} and \ref{ii} in \autoref{thm:main} are both sufficient and necessary for the BBM formula to hold.
We first address the sufficiency
by proving \autoref{cor:suff},
then we turn to the necessity, that is,
to \autoref{thm:main}.

\subsection{Sufficiency}\label{sec:suff}
As we outlined in \autoref{sec:intro},
\autoref{cor:suff} is an immediate consequence of the proof of \autoref{thm:compactness}. 

\begin{proof}[Proof of \autoref{cor:suff}]
    Under the current assumptions,
    we know that there exist
    an infinitesimal sequence of $\{\eps_k\}$ and
    two Radon measures $\mu\in \scrM(\Snmu)$ and $\nu\in \scrM(\RR^N)$
    such that \eqref{eq:cpt} is satisifed.

    In order to conclude,
    it now suffices to recall that
    the measure $\nu$ is the weak-$\ast$ limit
    of the sequence defined by \eqref{eq:nueps}
    (see Step 2 in the proof of \autoref{stm:cpt}).
    We are currently supposing that
    such sequence weakly-$\ast$ converges to $\alpha \delta_0$ for a suitable $\alpha\geq 0$:
    then, necessarily, $\nu=\alpha \delta_0$
    and the second integral on the right-hand side in \eqref{eq:cpt} vanishes.
    The conclusion is thus achieved.
\end{proof}

\subsection{Necessity}\label{subsec:nec}
We now focus on the proof of \autoref{thm:main},
thus showing that
the sufficient conditions devised in the previous subsection are also necessary
for the BBM formula to hold.
As before,
we rely on the formulation in Fourier variables provided by
\autoref{lemma:fourier}, or, in other words,
we assume that \eqref{eq:nonloc-to-loc-F} holds for every $\psi \in L^2_w(\RR^N)$ and for a given measure $\lambda\in \scrM({\mathbb{S}^{N-1}}).$
We first show that
such a nonlocal-to-local formula forces
the restrictions of the kernels $\{\rho_\eps\}$ to any large ball
to belong definitively to $L^1$,
while the integrals of $\rho_\eps(z)/|z|^2$ on the complement of such balls
need to become increasingly smaller
(see \eqref{eq:nec}).
Then, item \ref{ii} in \autoref{thm:main} will be derived as well.

\begin{proposition}\label{stm:concentration}
    Suppose that
   	the convergence in \eqref{eq:nonloc-to-loc-F} holds
    for every $\psi\in L^2_w(\RR^N)$ and for a given measure $\lambda\in \scrM({\mathbb{S}^{N-1}})$.
    Then, there exists $M\geq0$ depending only on $N$ and  $\lambda$
    such that for every $R>0$ condition \eqref{eq:nec} is satisfied.
\end{proposition}
\begin{proof}
    Throughout the proof, $c_N$ is a generic positive constant that depends just on the dimension $N$ and whose value may change from line to line.

    Let $\psi \in L^2_w(\RR^N)\mz$ be a radial function.
    Then, there exists a measurable $v\colon [0,+\infty) \to \RR$
    such that $\psi(\xi)=v(|\xi|)$ and that
    \begin{equation}\label{eq:integr-v}
    	0<\int_0^{+\infty} t^{N-1}(1+t^2) v^2(t) \mathd t <+\infty.
    \end{equation}
	We define
	\begin{gather*}
		\psi_R(\xi) \coloneqq R^{N/2} \psi ( R \xi )
		\qquad\text{for all } R>0,
	\end{gather*}
	and we observe that
	a change of variables yield
	\begin{gather}
		\int_{B(0,R^{-1})^c} |\psi_R(\xi)|^2\mathd \xi
		= \int_{B(0,1)^c} |\psi(\xi)|^2 \mathd \xi,
		\label{eq:psik} \\
		\int_{\RR^N} |\xi|^2 |\psi_R(\xi)|^2\mathd \xi =
		\frac{1}{R^2} \int_{\RR^N} |\xi|^2 |\psi(\xi)|^2\mathd \xi\nonumber.
	\end{gather}
	
	By choosing $\psi=\psi_R$
	in $\eqref{eq:nonloc-to-loc-F}$,
    we infer that
	\begin{align}
    \label{eq:poincare}
		\lim_{k\to+\infty}\int_{\RR^N} | \psi_R (\xi) |^2
			\int_{\RR^N} \rho_{\eps_k}(z)\frac{1 - \cos( z\cdot \xi)}{| z |^2}
		\mathd z
		\mathd \xi &
		\leq \lambda(\Snmu) \int_{\RR^N} |\xi|^2| |\psi_R(\xi)|^2\mathd \xi \\
		&\nonumber
    = \frac{c}{R^2},
	\end{align}
	where $c\coloneqq c(\lambda,\psi)$ is a suitable constant.
	We exchange the integrals on the left-hand side of \eqref{eq:poincare} by the Fubini theorem, and,
	for any fixed $z\in \RNmz$,
    recalling that $\hat{z}=z/|z|$,
	we let $L_{\hat z}$ be a rotation
    such that $\hat{z} = L_{\hat{z}}^{\mathrm{t}} e_1$,
    where the superscript $\mathrm{t}$ denotes transposition.
    A change of variables yields
    \begin{align}
    \label{eq:cov}
	\int_{\RR^N} 
			| \psi_R (\xi) |^2 \big( 1-\cos(z \cdot \xi) \big)
	\mathd \xi
	& = \int_{\RR^N} 
			| \psi_R (\xi) |^2 \Big( 1 - \cos( |z| e_1 \cdot (L_{\hat z}\xi) \big) \Big)
		\mathd \xi \\
	\nonumber & = \int_{\RR^N} 
			| \psi_R (\xi) |^2 \big( 1 - \cos( |z| e_1 \cdot \xi) \big)
		\mathd \xi
    \end{align}
    (recall that $\psi_R$ is radial).
    By plugging \eqref{eq:cov} into \eqref{eq:poincare},
    we obtain
    \begin{align*}
	    \lim_{k\to+\infty}\int_{\RR^N}
	    \frac{\rho_{\eps_k}( z)}{| z |^2} \int_{\RR^N} 
			| \psi_R (\xi) |^2  \big( 1 - \cos( |z| e_1 \cdot \xi) \big)
		\mathd \xi \mathd z \leq \frac{c}{R^2}.
    \end{align*}
    
    From now on, we detail the argument for $N\geq 4$ only;
    the lower dimensional cases may be addressed
    by similar (but lighter) computations.
    First, we change variables to find
    \begin{align}\label{eq:unif-bound2}
	    \lim_{k\to+\infty} \int_{\RR^N}
	    \frac{\rho_{\eps_k}( z)}{| z |^{N+2}}
	    \int_{\RR^N}
	    	\left| \psi_R \left( \frac{\xi}{| z |} \right) \right|^2
			\big( 1 - \cos( e_1 \cdot \xi) \big)
		\mathd \xi \mathd z
		\leq \frac{c}{R^2}.
    \end{align}
    Next, we rewrite the integral with respect to $\xi$ on the left-hand side of \eqref{eq:unif-bound2}
    by employing spherical coordinates:
    for $\sigma \in \Snmu$
    we consider $\theta_1,\dots,\theta_{N-2}\in [0,\pi]$ and $\theta_{N-1} \in [0,2\pi)$
    such that
    \begin{align*}
	    &e_1 \cdot \sigma = \cos(\theta_1), \\
	    &e_i \cdot \sigma = \cos(\theta_i)\prod_{j=1}^{i-1} \sin(\theta_j)
	    \qquad\text{for }i=2,\dots,N-1,\\
	    &e_N \cdot \sigma = \prod_{j=1}^{N-1} \sin(\theta_j).
    \end{align*}
    By the coarea formula,
    recalling that $\psi_R(\xi)=R^{N/2}v(R|\xi|)$ for $v$ as above, it holds
    \begin{align}
     \nonumber  &   \int_{\RR^N}
        	\left| \psi_R \left( \frac{\xi}{| z |} \right) \right|^2
			\big( 1 - \cos(e_1 \cdot \xi) \big)
	    \mathd \xi \\
	  \nonumber  &\quad  = R^N \int_0^{+\infty} v^2 \left( \frac{R}{| z | } t \right) t^{N-1} 
	        \int_{\Snmu}
			    \big( 1 - \cos(t e_1 \cdot \sigma) \big)
		    \mathd \mathscr{H}^{N-1}(\sigma) \mathd t \\
	\nonumber	&\quad = 
		    \int_0^{2\pi} \mathd \theta_{N-1}
		    \prod_{j=2}^{N-2} \int_0^{\pi}
		        \sin^{N-j-1}(\theta_{j})   
		    \mathd \theta_{j} \cdot\\
       \nonumber &\qquad\qquad \cdot R^N \int_0^{+\infty} v^2 \left( \frac{R}{| z | } t \right) t^{N-1}
		    \int_0^{\pi}
		        \big[
		        1 - \cos\big( t \cos(\theta_1) \big)
		        \big]
		        \sin^{N-2}(\theta_1)
		    \mathd \theta_1
		    \mathd t \\
	\nonumber	& \quad =
		    c_N
		   R^N \int_0^{+\infty} v^2 \left( \frac{R}{| z | } t \right) t^{N-1}
		    \int_0^{\pi}
		        \big[
		        1 - \cos\big( t \cos(\theta) \big)
		        \big]
		        \sin^{N-2}(\theta)
		    \mathd \theta
		    \mathd t \\
		\nonumber & \quad =
		    c_N
		    R^N \int_0^{+\infty} v^2 \left( \frac{R}{| z | } t \right) t
			\int_{-t}^{t}
				\big(1 - \cos(s) \big)
				(t^2 - s^2)^{\frac{N-3}{2}}
			\mathd s
			\mathd t 
    \end{align}
    Since the integrand in the last expression is positive,
    by restricting the domain of integration we find
    \begin{align}
        \int_{\RR^N}
                	\left| \psi_R \left( \frac{\xi}{| z |} \right) \right|^2
        			\big( 1 - \cos(e_1 \cdot \xi) \big)
        	    \mathd \xi
	    & \geq c_N
		    R^N \int_0^{+\infty} v^2 \left( \frac{R}{| z | } t \right) t^{N-2}
			\int_{-\frac{t}{2}}^{\frac{t}{2}}
				\big(1 - \cos(s) \big)
			\mathd s
			\mathd t \nonumber \\
		& \geq c_N
			R^N \int_0^{+\infty} v^2 \left( \frac{R}{| z | } t \right) t^{N-1}
		    \left(
		        1 - \frac{2}{t} \sin\left(\frac{t}{2}\right)
		    \right)
		    \mathd t. \label{eq:lower-bound-cos}
    \end{align}
    Next, we proceed by splitting the interval $(0,+\infty)$
    into two regions,
    and we analyze the corresponding integrals separately.
        
    We observe that
    by a Taylor expansion around $0$ there exists $\alpha_0>0$ such that
    \[
        1- \frac{2}{t} \sin\left(\frac{t}{2}\right) \geq \alpha_0 t^2
    	\qquad\text{for every } t\in (0,1].
    \]
    Then, starting from \eqref{eq:unif-bound2} and
    taking into account \eqref{eq:lower-bound-cos},
    we infer
    \begin{align*}
    \frac{c}{R^2} & \geq \limsup_{k\to+\infty}\int_{B(0,R)}
    		    \frac{\rho_{\eps_k}( z)}{|  z |^{N+2}}
    		    \int_{\RR^N}
    		    	\left| \psi_R \left( \frac{\xi}{| z |} \right) \right|^2
    				\big( 1 - \cos( e_1 \cdot \xi) \big)
    			\mathd \xi \mathd z
    			\\
    & \geq \alpha_0 c_N R^N \limsup_{k\to+\infty} \int_{B(0,R)}
        		    \frac{\rho_{\eps_k}( z)}{| z |^{N+2}}
        		    \int_{0}^{\frac{|z|}{R}}
        		    	t^{N+1} v^2 \left( \frac{R}{| z |} t \right) 
	    		    \mathd t \mathd z \\
	& = \frac{\alpha_0 c_N}{R^2} \limsup_{k\to+\infty} \int_{B(0,R)}
	        		    \rho_{\eps_k}( z )
	        		    \int_0^1
	        		    	t^{N+1} v^2 ( t )
		    		    \mathd t \mathd z.
    \end{align*}
	In conclusion, owing to \eqref{eq:integr-v}, we find
	\begin{equation*}
		\limsup_{k\to+\infty}\int_{B(0,R)} \rho_{\eps_k}( z ) \mathd z \leq M_0
	\end{equation*}
	for a suitable $M_0\coloneqq M_0(N,v)$.
    
    We now turn to the contribution accounting for `large' $|z|$.
    Note that there exists $\alpha_1>0$ such that
    \[
        1- \frac{2}{t} \sin\left(\frac{t}{2}\right) \geq \alpha_1
        \qquad\text{for every } t>1.
    \]
    Therefore, by estimates similar to the ones above we obtain
    \begin{align*}
         \frac{c}{R^2} & \geq \limsup_{k\to+\infty} \int_{B(0,R)^c}
            		    \frac{\rho_{\eps_k}( z)}{| z |^{N+2}}
            		    \int_{\RR^N}
            		    	\left| \psi_R \left( \frac{\xi}{| z |} \right) \right|^2
            				\big( 1 - \cos( e_1 \cdot \xi) \big)
            			\mathd \xi \mathd z
            			\\
            & \geq \alpha_1 c_N R^N \limsup_{k\to+\infty} \int_{B(0,R)^c}
                		    \frac{\rho_{\eps_k}( z)}{| z |^{N+2}}
                		    \int_{\frac{|z|}{R}}^{+\infty}
                		    	t^{N-1} v^2 \left( \frac{R}{| z |} t \right) 
        	    		    \mathd t \mathd z\\
    	& = \alpha_1 c_N \limsup_{k\to+\infty} \int_{B(0,R)^c}
    	        		    \frac{\rho_{\eps_k}( z)}{| z |^2}
    	        		    \int_1^{+\infty}
    	        		    	t^{N-1} v^2 ( t )
    		    		    \mathd t \mathd z, \nonumber
        \end{align*}
    and, again by \eqref{eq:integr-v}, we deduce
    \begin{equation*}
    	\limsup_{k\to+\infty}\int_{B(0,R)^c} \frac{\rho_{\eps_k}( z)}{| z |^2} \mathd z \leq \frac{M_1}{R^2}
    \end{equation*}
    for some $M_1\coloneqq M_1(N,v)$.
    
    To conclude the proof,
    we first optimize $M_0$ and $M_1$ with respect to $v$ and
    we choose as $M$ the largest of the two optima;
    note, in particular, that
    $M$ is finite and strictly positive, and depends only on the dimension of the space and on $\lambda(\Snmu)$.
\end{proof}

\begin{remark}\label{rmk:poincare}
    Observe that,
    heuristically,
    inequality \eqref{eq:poincare} has the same structure of a Poincar\'e inequality:
    the $L^2$-norm of a function
    on the left hand-side,
    the $L^2$-norm of its gradient on the right one.
    So, in a sense,
    the integral with respect to $z$ on the left hand-side may be regarded as the inverse of the Poincar\'e constant.
    The latter has a well-known scaling property:
    if $c_P(\Omega)$ denotes the Poincar\'e constant associated with a certain domain $\Omega$,
    then $c_P(R \Omega) = Rc_P(\Omega)$,
    where $R\Omega \coloneqq \{ x\in \RR^N : x/R \in \Omega \}$.
    Such considerations motivated
    the choice of the scaling of the test function $\psi$ in the proof above
    (recall that there we work in Fourier variables). 
\end{remark}

With \autoref{stm:concentration} at hand,
we are now in a position to prove \autoref{thm:main}.

\begin{proof}[Proof of \autoref{thm:main}]
Thanks to \autoref{stm:concentration},
we know that
item \ref{i} holds.
As a consequence,
there is an infinitesimal sequence $\{\eps_k\}$ such that the inequality in \eqref{eq:suff} holds,
and we may invoke the compactness result in \autoref{thm:compactness}.
Thus, there exist a subsequence $\{\eps_{k_n}\}$ and two Radon measures $\mu\in \scrM(\Snmu)$ and $\nu \in \scrM(\RR^N)$ such that
for every $u\in H^1(\RR^N)$
	\begin{equation*}
		\lim_{n\to +\infty} \scrF_{\eps_{k_n}}[u]
		= \frac{1}{2}\int_{\RR^N}
		\left[
		    \int_{\Snmu} | \nabla u(x) \cdot \sigma |^2 \mathd \mu(\sigma)
		    +
		    \int_{\RR^N\mz}
				\frac{| u (x+z) - u (x) |^2}{| z |^2} \mathd \nu(z)
		\right]
		\mathd x.	    
	\end{equation*}
In particular, from the proof of \autoref{thm:compactness}
we know that $\nu$ is the weak-$\ast$ limit in $\scrM_{\rm loc}(\RR^N)$ of the subsequence $\{\nu_{k_n}\}$ defined by
    \begin{equation}\label{eq:nu-kn}
	 		\langle \nu_{k_n} , f \rangle
	 			\coloneqq  \int_{\RR^N} \rho_{\eps_{k_n}}(z) f(z) \mathd z
	 		\qquad\text{for all } f\in C_c(\RR^N).
    \end{equation}
Note that, in principle, the measures $\mu$ and $\lambda$ may differ.
However, since we are assuming \eqref{eq:BBM},
for every $u\in H^1(\RR^N)$ it must hold
\begin{multline*}
		\int_{\RR^N} \int_{\Snmu} | \nabla u(x) \cdot \sigma |^2 \mathd \lambda(\sigma) \mathd x
		\\= \frac{1}{2}\int_{\RR^N}
		\left[
		    \int_{\Snmu} | \nabla u(x) \cdot \sigma |^2 \mathd \mu(\sigma)
		    +
		    \int_{\RR^N\mz}
				\frac{| u (x+z) - u (x) |^2}{| z |^2} \mathd \nu(z)
		\right]
		\mathd x.	    
	\end{multline*}


    By passing to Fourier variables as in the proof of \autoref{lemma:fourier},
    the previous equality becomes
	\begin{multline*}\label{eq:mu-nu-lambda}
    \int_{\RR^N} |\psi (\xi) |^2
					\int_{\Snmu} | \xi \cdot \sigma |^2 \mathd \lambda(\sigma)
					\mathd \xi \\
     =
	\int_{\RR^N} |\psi(\xi)|^2 
		\left[
		\frac{1}{2}
			\int_{\Snmu} |\xi \cdot \hat z|^2
		\mathd\mu(z)
		+
		 \int_{\RNmz}
			\frac{1-\cos( z\cdot \xi)}{|z|^2}
		\mathd\nu(z)
		\right]
	\mathd \xi
    \end{multline*}
	for every $\psi\in L^2_w(\RR^N)$,
    whence,
	by the fundamental theorem of the calculus of variations and the continuity with respect to the $\xi$ variable, we deduce
	\begin{equation}\label{eq:mu-nu-lambda}
        \int_{\Snmu} | \xi \cdot \sigma |^2 \mathd \lambda(\sigma)
        =
		\frac{1}{2} \int_{\Snmu}
			|\xi \cdot \hat z|^2
		\mathd\mu(z)
		+
		 \int_{\RR^N}
			\frac{1-\cos( z\cdot \xi)}{|z|^2}
		\mathd\nu(z)
		\qquad\text{for every }\xi\in\RR^N.
	\end{equation}
	Then, by dividing \eqref{eq:mu-nu-lambda} by $|\xi|^2$ and
	letting $|\xi|\to+\infty$, we obtain
	\begin{equation}\label{eq:lambda-vs-mu}
	   \int_{\Snmu} | \hat \xi \cdot \sigma |^2 \mathd \lambda(\sigma)
        =
        \frac{1}{2} \int_{\Snmu}
		 	|\hat \xi \cdot \hat z|^2
		  \mathd\mu(z)
        \qquad\text{for every }\hat\xi\in\Snmu.
	\end{equation}        
	It follows that necessarily
	\[
	\int_{\RNmz}
				\frac{1-\cos( z\cdot \hat\xi)}{|z|^2}
	\mathd\nu(z) = 0
	\qquad\text{for every }\hat\xi\in\Snmu,
	\]
	but since $z\mapsto (1-\cos( z\cdot \hat\xi))/|z|^2$ is a positive function with support on the whole space for every $\xi$,
	we infer that
    the restriction of $\nu$ to $\RR^N\mz$ is $0$.
    By the definition of Lebesgue integral,
    we obtain that for any $f\in C_c(\RR^N)$
    \[
        \int_{\RR^N} f(z) \mathd \nu(z)
        = \nu(\{0\}) f(0),    
    \]
    that is, $\nu=\alpha \delta_0$
    for a suitable $\alpha\geq 0$.
    
    Finally, we conclude the proof of item \ref{ii}
    by observing that
    for any subsequence $\{\eps_{k_n}\}$
    the associated sequence of measures $\{\nu_{k_n}\}$ defined by \eqref{eq:nu-kn} must converge weakly-$\ast$ to $\alpha\delta_0$, and
    hence the whole sequence $\{\nu_k\}$ converges.
\end{proof}

\begin{remark}
\label{rk:quad}
   For each $\lambda\in \scrM(\mathbb{S}^{N-1})$,
   let us define the positive semi-definite symmetric matrix
   \[
    A_\lambda \coloneqq \int_{\Snmu} \sigma \otimes \sigma \mathd \lambda(\sigma).
   \]
   By employing this notation,
   the functional $\scrF$ in \eqref{eq:limit-func} rewrites as
    \begin{equation*}
        \scrF[u]= \int_{\RR^N} A_\lambda\nabla u\cdot\nabla u\mathd x
    \end{equation*}
    for every $u\in H^1(\mathbb{R}^N)$.
    
    As we observed in the previous proof, under the assumptions of \autoref{thm:main}
    the measure $\lambda$ in \eqref{eq:BBM} and the measure $\mu$ obtained by the compactness argument need not be the same.
    However, equality \eqref{eq:lambda-vs-mu} expresses the fact that the associated matrices $A_\lambda$ and $A_\mu$ do coincide.
\end{remark}


\section{Discussion and perspectives}\label{sec:final}
In what follows,
we first present an alternative formulation
of condition \ref{i} in \autoref{thm:main}, and
we then compare our results with previous ones in other contributions.
In particular, we explain
how some classes of kernels that have been considered in the literature  are encompassed by our theory.
We conclude by outlining possible future investigations.

\subsection{L\'evy conditions and reformulation of \ref{i}}
\label{subs:cond1}
As we recalled in \autoref{sec:intro},
the research on nonlocal-to-local formulas has been focused on sufficient conditions.
It must be however mentioned that
necessary conditions for the finiteness of the nonlocal energies in \eqref{eq:energy} have been devised as well, and
they are sometimes referred to as {\em L\'evy conditions}.
It is indeed known that,
when $u\in H^1(\RR^N)$,
an $\eps$-uniform upper bound on the functionals in \eqref{eq:energy}
entails a certain summability close to the origin and a decay at infinity.
Precisely, the following can be shown:

	\begin{theorem}\label{thm:levy}
	Suppose that for every $u\in H^1(\RR^N)$
	there exists $c\coloneqq c(u)\geq0$ such that
	$\scrF_\eps[u]\leq c$ for all $\eps\in J$.
	Then, the family $\{\rho_\eps\}$ fulfils the L\'evy conditions, that is,
	there exists $M\geq0$ such that
		\begin{equation*}
			\int_{B(0,1)} \rho_\eps(z) \mathd z
				+ \int_{B(0,1)^c} \frac{\rho_\eps(z)}{| z |^2} \mathd z
			\leq M
		\qquad\text{for every $\eps\in  J$.}
		\end{equation*}
	\end{theorem}
    For a proof,
    we refer, e.g.,~to the recent contribution \cite[Thm.~2.1]{FK22}
    (the authors work under radiality assumptions on the kernels,
    but for the result at stake this does not play a role).
    Alternatively,
    we note that the argument in the proof of \autoref{stm:concentration}
    may be adapted to establish the previous proposition:
    it is enough to work
	with a fixed test function $\psi\in L^2_w(\RR^N)$.

When the bound in \autoref{thm:levy} holds only asymptotically, that is,
$\limsup_{\eps\to 0} \scrF_\eps[u]\leq c$,
it can be shown that
    \begin{gather}\label{eq:levy2}
    \limsup_{\eps\to 0}
    \int_{\RR^N}\frac{\rho_{\eps}(z)}{1+|z|^2}\mathd z\leq M.
    \end{gather}
Such bound is necessary, but not sufficient
for the one in \ref{i}:
as a counterexample, consider
for $N=1$ the constant family $\rho_\eps\equiv 1$.
As we observed in \autoref{sec:intro}, indeed,
condition \ref{i} may be regarded as a uniform decay requirement on the kernels.
In more precise terms, the following holds:

    \begin{lemma}\label{prop:eq}
    Condition \ref{i} is equivalent to the following:
    \begin{itemize}
    \item[(i')]\label{i'} There exists $\tilde M\geq 0$ such that for every $R>0$ there holds
    \begin{gather}\label{eq:i-alt}
        \limsup_{k\to+\infty}
        \int_{\RR^N}\frac{\rho_{\eps_k}(z)}{R^2+|z|^2}\mathd z
        \leq\frac{\tilde M}{R^2}.
    \end{gather}
    \end{itemize}
    \end{lemma}
    \begin{proof}
We first show that \eqref{eq:nec} implies \eqref{eq:i-alt}. Fix $R>0$. After a change of variable, \eqref{eq:nec} rewrites as
\begin{gather*}
            \limsup_{k\to+\infty}\left[
            \int_{B(0,1)} \rho_{\eps_k}(Rz)\mathd z
            +
             \int_{B(0,1)^c} \frac{\rho_{\eps_k}( Rz)}{| z |^2} \mathd z
            \right]
            \leq \frac{M}{R^N}.
\end{gather*}
The conclusion follows then by observing that
$$\int_{B(0,1)} \rho_{\eps_k}(Rz)\mathd z
            +
             \int_{B(0,1)^c} \frac{\rho_{\eps_k}( Rz)}{| z |^2} \mathd z\geq \int_{\RR^N}\frac{\rho_{\eps_k}(Rz)}{1+|z|^2}\mathd z
            $$
 and by performing a further change of variables.  

 Conversely, assume that \eqref{eq:i-alt} holds. Then, for every $R>0$ a change of variable yields
 $$\limsup_{k\to +\infty}\int_{\RR^N}\frac{\rho_{\eps_k}(Rz)}{1+|z|^2}\mathd z\leq \frac{\tilde M}{R^N}.$$
 Since the real function $t\mapsto {t^2}/{(1+t^2)}$ is increasing on the positive real line, we find
 \begin{align*}
     \int_{\RR^N}\frac{\rho_{\eps_k}(Rz)}{1+|z|^2}\mathd z&\geq\frac12\int_{B(0,1)}\frac{\rho_{\eps_k}(Rz)}{1+|z|^2}\mathd z+\int_{B(0,1)^c}\frac{|z|^2}{1+|z|^2}\frac{\rho_{\eps_k}(Rz)}{|z|^2}\mathd z\\
     &\geq \frac{1}{2}\left(\int_{B(0,1)} \rho_{\eps_k}(Rz)\mathd z
            +
             \int_{B(0,1)^c} \frac{\rho_{\eps_k}( Rz)}{| z |^2} \mathd z\right).
 \end{align*}
 A further change of variable entails \eqref{eq:nec}.
\end{proof}

\begin{remark}
We observed that \eqref{eq:levy2} is necessary for \eqref{eq:nec} to hold.
On the other hand,
a sufficient condition not involving the parameter $R$ is the following: there exists
an infinitesimal family $\{\omega_\eps\} \subset (0,+\infty)$ such that 
    \begin{gather}\label{eq:omegaeps}
    \limsup_{\eps\to 0}
        \int_{\RR^N}\frac{\rho_{\eps}(z)}{1+\omega_{\eps}|z|^2}\mathd z
    < +\infty.
    \end{gather} 
This condition is however stronger than \ref{i}:
to see this, given a family $\{\omega_\eps\}$ as above,
observe that for $N=1$
the kernels $\rho_\eps(z) \coloneqq \omega_\eps^{1/4}$ fulfil \eqref{eq:nec}, but not \eqref{eq:omegaeps}.
\end{remark}

\subsection{$L^1$ and fractional kernels}
In \cite{BBM01} the authors proved their nonlocal-to-local formula
under the assumption that
the kernels $\rho_\eps$ are standard mollifiers.
A more general version of their result is the following:

\begin{theorem}[cf. Thm. 1 in \cite{Pon04}]\label{thm:ponce}
	Let $p\in(1,+\infty)$ be fixed.
	For every $\eps\in J$,
	let $\rho_\eps\colon \RR^N\to [0,+\infty)$ be a function
	with $\| \rho_\eps/2 \|_{L^1(\RR^N)}=1$.
	Suppose also that for every $\delta>0$
	\begin{equation}\label{eq:concentr}
		\lim_{\eps\to 0} \int_{B(0,\delta)^c} \rho_\eps(z) \mathd z = 0.
	\end{equation}
	Then, for any $u\in W^{1,p}(\RR^N)$
	there exists $c>0$ such that
	\begin{equation*}
		\int_{\RR^N \times \RR^N}
					\rho_\eps (x-y)
					\frac{| u (x) - u (y) |^p}{| x - y |^p} 
				\mathd y \mathd x \leq c
		\quad\text{for every } \eps\in J.
	\end{equation*}
	Besides, there exist
	an infinitesimal sequence $\{\eps_k\}\subset J$ and
	a positive Radon measure $\lambda$ on the unit sphere $\Snmu$
	that depends only on $\{\rho_{\eps}\}$
	such that $\int_{\Snmu} \mathd\lambda = 1$ and 
	\begin{equation}\label{eq:ponce}
			\lim_{k\to+\infty} \frac{1}{2} \int_{\RR^N \times \RR^N}
						\rho_{\eps_k}(x-y)
						\frac{| u (x) - u (y) |^p}{| x - y |^p} 
					\mathd y \mathd x
			= \int_{\RR^N} \int_{\Snmu} | \nabla u(x) \cdot \sigma |^p \mathd \lambda(\sigma) \mathd x
	\end{equation}
	for every $u\in W^{1,p}(\RR^N)$.
\end{theorem}

We now show how
the class of kernels considered in the theorem above falls within our theory.

\begin{example}[$L^1$ kernels]\label{stm:compar}
    Let $\{\rho_\eps\}_{\eps\in J}$ be a family of kernels as in \autoref{thm:ponce}.
    A direct check shows that
    the normalization condition implies \eqref{eq:suff}.
    Besides, for every $f\in C_c(\RNmz)$ 
    there exists $\delta>0$ so small that
    \[
        \int_{\RNmz} \rho_\eps(z)f(z) \mathd z
        = \int_{B(0,\delta)^c} \rho_\eps(z)f(z) \mathd z.
    \]
    It hence follows from \eqref{eq:concentr} that
    \[
        \lim_{\eps\to 0}\int_{\RNmz} \rho_\eps(z)f(z) \mathd z
        = 0,
    \]
    which entails, similarly to the proof of \autoref{cor:suff}, that
    the weak-$\ast$ limit of the associated sequence in \eqref{eq:nueps} is a multiple of $\delta_0$.
\end{example}

As we commented in \autoref{sec:intro},
fractional kernels are not exactly covered by \autoref{thm:ponce}.
With the next example,
we see how they fit in our framework.

\begin{example}[Fractional kernels]\label{stm:frac}
    Given $s\in (0,1)$ and $u\in L^2(\RR^N)$,
    the (normalised) {\em $s$-Gagliardo seminorm} of $u$ is defined by
    \begin{equation*}
        \scrG_s[u]\coloneqq
        \frac{1-s}{2}\int_{\RR^N\times \RR^N}
        \frac{| u (x) - u (y) |^2}{| x - y |^{N+2s}}\mathd y  \mathd x.
    \end{equation*}
    Such functional corresponds to the one in \eqref{eq:energy}
    upon selecting
    \begin{gather*}
        \eps \coloneqq 1-s,
        \qquad
        \rho_\eps(z)=\rho_\eps^{\scrG}(z) \coloneqq \frac{\eps}{2|z|^{N-2\eps}}.
    \end{gather*}
    Note that in this case $\rho_\eps\notin L^1(\RR^N)$. 
    On the other hand, for every $\delta>0$ and
    for suitable $N$-depending constants $\alpha_0,\alpha_1>0$,
    we have
    \begin{gather*}
        \int_{B(0,\delta)} \frac{\eps}{2|z|^{N-2\eps}} \mathd z = \alpha_0 \delta^{2\eps}, \\
        \int_{B(0,\delta)^c} \frac{\eps}{2|z|^{N-2\eps+2}} \mathd z = \alpha_1\frac{\eps}{(1-\eps)\delta^{2(1-\eps)}}.
    \end{gather*}
    In particular, by taking, e.g.,~$M=2$,
    we see that \eqref{eq:suff} holds.
    Besides, for every $R>\delta>0$ we have
    \begin{gather*}
        \lim_{\eps\to 0} \int_{B(0,R)\setminus B(0,\delta)}
        \frac{\eps}{2|z|^{N-2\eps}} \mathd z
        = \alpha_0 \lim_{\eps\to 0} (R^{2\eps}-\delta^{2\eps})
        = 0,
    \end{gather*}
    whence, similarly to the previous example,
    we infer that
    $\{\rho_\eps^{\scrG}\}$ converges locally weakly-$\ast$ to a multiple of the Dirac delta in $0$
    in the sense of Radon measures.
\end{example}

\subsection{Future directions}\label{subsec:comments}
In this paper we provided sufficient and necessary conditions on a family of kernels $\{\rho_\eps\}$
for the nonlocal functionals in \eqref{eq:energy} to converge to a variant of the Dirichlet integral
for every $u\in H^1(\RR^N)$.
It is natural to wonder
whether such characterization still holds for the more general functionals
considered in \cite{BBM01}.
We conjecture that this is the case.
Namely,
given a family of positive, measurable kernels $\{\rho_\eps\}_{\eps\in J}$,
we conjecture that
for any open set $\Omega\subseteq \RR^N$ with Lipschitz boundary and for any $p\in [1,+\infty)$
the following conditions are necessary and sufficient for the BBM formula to hold
for every $u\in W^{1,p}(\Omega)$ when $p>1$
or $u\in BV(\Omega)$ when $p=1$:
\begin{enumerate}[label=\emph{(\roman*)}]
    \item there exists $M\geq0$ such that
	   	for every $R>0$ it holds
        \begin{gather*}
	    \limsup_{\eps\to0}\int_{B(0,R)} \rho_\eps(z)\mathd z \leq M, \\
        \limsup_{\eps\to0} \int_{B(0,R)^c} \frac{\rho_\eps( z)}{| z |^p} \mathd z \leq  \frac{M}{R^p}
         \qquad\text{when $\Omega$ is unbounded};
        \end{gather*}
	 \item there exists an infinitesimal sequence $\{\eps_k\}\subset J$ such that
    the sequence of measures $\{\nu_k\}\subset\scrM_{\rm loc}(\RR^N)$ defined as in \eqref{eq:nuk}
    converges locally weakly-$\ast$ to $\alpha \delta_0$ in the sense of Radon measures for a suitable $\alpha \geq 0$.
\end{enumerate}

We remind that it is known that
the BBM formula fails
when the boundary of $\Omega$ is not regular enough
(see~\cite[Rmk. 1]{Pon04}, and \cite{LS11} on a possible remedy).

Naturally, for $p\neq 2$ and $\Omega \subsetneq \RR^N$
the Fourier approach is not viable anymore
(but when $p\neq 2$ and $\Omega = \RR^N$
techniques of Fourier analysis may still be invoked
by resorting to the Littlewood-Paley theory,
as it is done in the recent contribution \cite{BSY22}).
A possible strategy
to establish the necessity of the previous conditions
is to follow the proof of \cite[Thm.~2.1]{FK22}
and employ rescaled test functions as in the proof of \autoref{stm:concentration}.

A second research direction concerns the variational convergence of the nonlocal energies to local ones,
in the same spirit as \cite[Thm. 8 and Cor. 8]{Pon04}.
For a thorough treatment of $\Gamma$-convergence we refer to the monograph \cite{DM93}.
It is not difficult to see that
the conditions in \autoref{cor:suff} are {\em sufficient} for the $\Gamma$-convergence of $\{\hat{\scrF}_\eps\}$ to $\hat{\scrF}$ when it is known that the limiting function $u$ has Sobolev regularity;
under this extra assumption,
by the inverse Fourier transform,
the $\Gamma$ convergence of $\{\scrF_\eps\}$ to $\scrF$ is recovered.
Proving that they are also {\em necessary} would require a refinement of \autoref{stm:concentration},
again possibly resorting to the approach of \cite[Thm.~2.1]{FK22};
note, in particular, that 
in our analysis \eqref{eq:nec} is derived from a $\Gamma$-limsup type inequality (see~\eqref{eq:poincare}).

$\Gamma$-convergence results are usually complemented
by equi-coercivity statements,
because in this way
convergences of minima and minimizers are obtained
thanks to the so-called fundamental theorem of $\Gamma$-convergence,
see~e.g.~\cite[Cor. 7.20]{DM93}.
Such results also have a role in devising the domain of the $\Gamma$-limit.
The convergence properties
of sequences of $L^p$ functions
with equi-bounded nonlocal energy
were considered already in \cite[Thm. 4]{BBM01};
refined results in the same vein have been obtained in \cite[Thm. 1.2 and 1.3]{Pon04-2} and,
more recently, in \cite[Thm. 4.2]{AAB22}.
Another natural question that is left open from our analysis is
what conditions on the kernels $\{\rho_\eps\}$ are necessary and sufficient for such a compactness result to hold.
It is expected that some requirement on the support of the measures $\mu$ in \autoref{thm:compactness} has to be enforced
(cf.~\cite[Thm. 5]{Pon04} and \cite[Thm. 3.1]{AAB22}).

\section*{Acknowledgments}
The research presented in this paper benefited from the participation of V.P. in the workshop {\it Nonlocality: Analysis, Numerics and Applications}
held at the Lorentz Center (Leiden, Netherlands) on 4--7 October 2022.
V.P. is a member of INdAM-GNAMPA.
E.D. and V.P. acknowledge support
by the Austrian Science Fund (FWF) through projects F65, V 662, Y1292, and I 4052, as well as from OeAD through the WTZ grants CZ02/2022 and CZ 09/2023. 
G.~Di~F. acknowledges support from the Austrian Science Fund (FWF)
through the project {\emph{Analysis and Modeling of Magnetic Skyrmions}} (grant P-34609). G.~Di~F. thanks the Hausdorff Research Institute for Mathematics in Bonn for its hospitality during the Trimester Program \emph{Mathematics for Complex Materials}. G.~Di~F. also thanks TU Wien and MedUni Wien for their support and hospitality. 



{\normalsize
	\bibliography{bib-davoli-difratta-pagliari}
}
\end{document}